\documentclass[reqno]{amsart}
 \usepackage{hyperref}
 \newtheorem{theorem}{Theorem}[section]

 \newtheorem{lemma}[theorem]{Lemma}

\theoremstyle{definition}
\newtheorem{definition}[theorem]{Definition}
\newtheorem{example}[theorem]{Example}

\newtheorem{corollary}[theorem]{Corollary}
\newtheorem{proposition}[theorem]{Proposition}
\theoremstyle{remark}
\newtheorem{remark}[theorem]{Remark}

\numberwithin{equation}{section}

\DeclareMathOperator{\AP}{AP}

\DeclareMathOperator{\re}{Re}

\begin{document}
\def\C{\mathbb C}
\def\R{\mathbb R}
\def\X{\mathbb X}
\def\cA{\mathcal A}
\def\cT{\mathcal T}
\def\Z{\mathbb Z}
\def\Y{\mathbb Y}
\def\Q{\mathbb Q}
\def\G{\mathcal G}
\def\Z{\mathbb Z}
\def\N{\mathbb N}
\def\M{\mathcal M}
\def\cal{\mathcal}

\def\cD{\cal D}
\def\tD{\tilde{{\cal D}}}
\def\F{\cal F}
\def\tf{\tilde{f}}
\def\tg{\tilde{g}}
\def\tu{\tilde{u}}

\title[Linear Almost Periodic Differential Equations]{Asymptotic Behavior of Linear\\ Almost Periodic Differential Equations}
\author{Bui Xuan Dieu}
\address{Department of Mathematics, Dresden University of Technology, 01062 Dresden, Germany\\ School of Applied Mathematics \& Informatics, Hanoi University of Science and Technology, 01 Dai Co Viet road, Ha Noi, Viet Nam}
\email{Bui.Xuan\_Dieu@mailbox.tu-dresden.de, dieu.buixuan@hust.edu.vn}
\author{Luu Hoang Duc}
\address{Department of Mathematics, Dresden University of Technology, 01062 Dresden, Germany\\ Institute of Mathematics, Vietnam Academy of Science and Technology, 18 Hoang Quoc Viet road, 10307 Ha Noi, Viet Nam}
\email{Hoang\_Duc.Luu@tu-dresden.de, lhduc@math.ac.vn}
\author{Stefan Siegmund}
\address{Department of Mathematics, Dresden University of Technology, 01062 Dresden, Germany}
\email{stefan.siegmund@tu-dresden.de}
\author{Nguyen Van Minh}\thanks{The third author (N.V.M) is the corresponding author}
\address{Department of Mathematics, Columbus State University, 4225 University Avenue,
Columbus, GA 31907. USA} \email{nguyen\_minh2@columbusstate.edu}

%    General info


%%%%%%%%%%%%%% ABSTRACT %%%%%%%%%%%%%%%%%%%%%%%%%%%%%%%
\begin{abstract}
The present paper is concerned with strong stability of solutions of non-autonomous equations of the form $\dot u(t)=A(t)u(t)$, where $A(t)$ is an unbounded operator in a Banach space depending almost periodically on $t$. A general condition on strong stability is given in terms of Perron conditions on the solvability of the associated inhomogeneous equation.
\end{abstract}
\date{\today}

\maketitle

%%%%%%%%%%%%%% INTRODUCTION %%%%%%%%%%%%%%%%%%%%%

\section{Introduction}
In this paper we consider the asymptotic behavior of solutions of evolution equations of the form
\begin{equation}\label{eq0}
\frac{dx}{dt}= A(t)x, \quad t \in \R ,
\end{equation}
in terms of the existence and uniqueness of bounded solutions to the inhomogeneous equations
\begin{equation}\label{eq}
\frac{dx}{dt}= A(t)x+f(t), \quad t \in \R ,
\end{equation}
where $A(t)$ is a family of unbounded linear operators on a (complex) Banach space $\X$ that depends almost periodically on $t$, and $f$ is an almost periodic function taking values in $\X$. Throughout the paper we assume that the homogeneous equation (\ref{eq0}) generates an almost periodic evolutionary process $(U(t,s)_{t\ge s})$ (see Definition \ref{def evpr}).
If  (\ref{eq0}) is autonomous (that is, $A(t)=A$ for all $t$), and $\dim (\X)<\infty$ the classical Lyapunov Theorem states that (\ref{eq}) is strongly stable if all real parts of the eigenvalues of $A$ are negative. In the infinite dimensional case this condition is no longer valid. In fact, one needs more conditions to guarantee the (strong) stability of solutions to (\ref{eq0}). We refer the reader to \cite{arebathieneu,bas,chitom,min3,nee} and the references therein. If $A(t)$ depends on $t$, the spectra $\sigma (A(t))$, in general, do not play any role in determining the behavior of solutions to (\ref{eq0}). If $A(t)$ depends periodically on $t$ with period $\tau$, the period map $P:=U(\tau ,0)$ can be used to study the problem via discrete analogs, results that can be found in \cite{arebat,min,vu}.

\bigskip
The problem becomes much more complicated when $A(t)$ depends on $t$ almost periodically (but not periodically). The idea of linear skew products has been used extensively to study the stability and exponential dichotomy of non-autonomous equations (see e.g. \cite{elljoh,johsel,sacsel} and the references therein). In this direction, the concept of evolution semigroups, as a variation of the aforementioned idea, proves to be a very effective analytic tool
 (see \cite{balgolpar,chilat,minrabsch}). However, since typically evolution semigroups are considered in the function spaces $L^p(\X), C_0(\X)$ or $\AP(\X)$, the spectrum of the evolution semigroup associated with (\ref{eq0}) is too coarse to characterize finer properties of the system like strong stability. Indeed, the spectrum of the generator of the evolution semigroup associated with an evolution equation in one of these function spaces consists of a union of vertical strips in the complex plane, hence the imaginary axis is either contained completely in the spectrum or does not intersect it. On the other hand, the well known ABLV Theorem (see Theorem \ref{the ABLV}) for stability of $C_0$-semigroups allows the generator's spectrum to intersect the imaginary axis.

The main purpose of this paper is to provide a setting in which the idea of evolution semigroups combined with the spectral theory of functions can be further used to study the asymptotic behavior of solutions of (\ref{eq0}). We consider the evolution semigroup associated with (\ref{eq0}) in the smallest invariant subspace of the space of all almost periodic functions $\AP(\X)$ which we call {\it minimal evolution semigroup} of (\ref{eq0}). In general, this smallest invariant function space is determined by the Bohr spectrum of the coefficient operator $A(t)$. The main results (cf.\ Theorem \ref{the main}) we obtain in the paper are extensions of results known in the autonomous and periodic cases. Our conditions for strong stability  are stated in terms of Perron conditions which are very popular in recent studies on stability and dichotomy (see for example \cite{huy1,huy2,minrabsch,pre,pre2,pre3}). We analyze some particular cases as examples of how the obtained results can be applied to equations with almost periodic coefficients.

%%%%%%%%%%%%%% PRELIMINARIES %%%%%%%%%%%%%%%%%%%%%%%%%%%%%%%

\section{Preliminaries}
\subsection{ Almost Periodic Functions}
In this paper we use the concept of almost periodicity in Bohr's sense. The reader is referred to \cite{arebathieneu, levzhi} for some standard definitions and properties of almost periodic functions taking values in a Banach space $\X$.

\begin{definition}
A bounded and continuous function $g:\R \to \X$ is said to be almost periodic in the sense of Bohr (or simply almost periodic) if for each given sequence $\{\tau_n\}_{n=1}^\infty \subset \R$ there exists a subsequence $\{ \tau_{n_k}\}_{k=1}^\infty$ such that the limit
\begin{equation*}
\lim_{k\to\infty} g(t+\tau_{n_k})
\end{equation*}
exists uniformly in $t\in\R$.
\end{definition}
Given an almost periodic function $g$, for each $\lambda \in \R$ the following is shown to exist (see \cite{levzhi})
\begin{equation*}
M_{\lambda,g}:=\lim_{T\to\infty}\frac{1}{2T}\int^T_{-T} e^{-i\lambda t}g(t)dt.
\end{equation*}
And, except for at most a countable set $\sigma_b(g)$ of values of $\lambda$, this limit $M_{\lambda,g}$ is always equal to zero.

\begin{definition}
Let $T\subset \R$. The semi-module generated by $T$, denoted by $sm (T)$, is the set of all real numbers $\mu$ of the form
\begin{equation*}
\mu:= n_1 \lambda_1 +n_2\lambda_2 +\cdots +n_k\lambda_k
\end{equation*}
where $k$ is a positive integer, $\lambda_1,\lambda_2, \dots ,\lambda _k \in T$, and $n_1,n_2,\dots ,n_k$ are non-negative integers.
\end{definition}
Note that by this definition $0\in sm (T)$.

\bigskip
Let $\sigma_b (A)$ be the Bohr spectrum of the almost periodic $A:\R \to \X$. We will denote the semi-module generated by this spectrum by $\Lambda := sm(  \sigma_b (A)  )$. We introduce the following notation
\begin{equation*}
\AP_\Lambda (\X):=\{ g\in \AP(\X): \sigma_b (g) \subset \Lambda \}.
\end{equation*}
Note that $\AP_\Lambda (\X)$ is a closed subspace of $\AP(\X)$.

%%%%%%%%%%%%%%%%%%%%%%%%%%%%%%%%%%%%%%%%%%%%%%%%%%%
Let $A$ be a closed operator in a Banach space $\X$, and let $A$ generate a uniformly bounded $C_0$-semigroup of linear operators $(T(t))_{t\ge 0}$, i.e., $\sup_{t\ge 0} \| A(t)\| <\infty$.  The following lemma is proved in \cite[p. 2073]{batneerab}.
\begin{lemma}\label{added}
Let $x \in \X$ be fixed, then the map
\[
R: \{Re \lambda > 0 \} \to \X, \lambda \mapsto R(\lambda , A)x
\]
is holomorphic. Furthermore, denote by $\sigma_u (A, x)$ {\em the local unitary spectrum} of $x$, i.e.\ the set of points $\lambda \in i\R$ to which $R$ cannot be extended holomorphically. Then $\sigma_ u (A, x) \subset \sigma (A) \cap i\R$, and
\[
\sigma (A) \cap i \R = \underset{x \in \X}{\cup} \sigma_u (A, x).
\]
\end{lemma}

%%%%%%%%%%%%%%%%%%%%%%%%%%%%%%%%%%%%%%%%%%%%%%%%%%%
We also restate a well known result from \cite[Theorem 3.4]{batneerab}.

\begin{theorem}[ABLV Theorem]\label{the ABLV}
Suppose $(T(t))_{t \ge 0}$ is a bounded $C_0$-semigroup in a Banach space $\X$ with generator $A$, $x \in \X$ is fixed. Denote by $\sigma_u(A, x)$
the set of $i \beta \in i\R$ such that the local resolvent, $R(\alpha + i\beta, A)x, \alpha > 0$, does not extend
analytically in some neighbourhood of $i\beta$. If
\begin{itemize}
\item [(i)] $\sigma _u(A, x)$ is countable,
\item [(ii)]$ \lim\limits_ {\alpha \downarrow 0} \alpha R(\alpha + i\beta, A)x = 0,$  for all $\beta$ with $i\beta \in \sigma_u(A, x)$,
\end{itemize}
then
\[
\lim\limits_{t \to \infty} \| T(t)x \|  = 0.
\]
\end{theorem}

%%%%%%%%%%%%%%%%%%%%%%%%%%%%%%%%%%%%%%%%%%%%%%%%%%%

%%%%%%%%%%%%%%%%%%%%%%%%%%%%%%%%%%%%%%%%%%%%%%%%%%%

\subsection{Evolution semigroups associated with evolutionary processes}

\begin{definition}\label{def evpr}
A two parameter family $(U(t,s))_{t\ge s}$ of bounded linear operators acting in a Banach space $\X$ is said to be an \emph{evolutionary process} if the following conditions are satisfied:
\begin{itemize}
\item[(i)] $U(t,t)=Id$, for all $t\in\R$, where $Id$ is the identity operator of $\X$,
\item[(ii)] $U(t,r)U(r,s)=U(t,s)$ for all $s\le r\le t$,
\item[(iii)] There are non-negative numbers $M,\alpha$ such that $\| U(t,s)\| \le Me^{\alpha (t-s)}$ for all $t\ge s$,
\item[(iv)] The map $(t,s)\mapsto U(t,s)x$ is continuous for each $x\in\X$.
\end{itemize}
An evolutionary process $(U(t,s))_{t\ge s}$ in a Banach space $\X$ is said to be \emph{almost periodic} if
\begin{itemize}
\item[(v)] for each $x\in \X, s\in \R$ the function $\R\ni t \mapsto U(t+s,t)x\in \X$ is almost periodic.
\end{itemize}
\end{definition}

%%%%%%%%%%%%%%%%%%%%%%%%%%%%%%%%%%%%%%%%%%%%%%%%%%%

\begin{definition}\label{evolsmg}
Given a function space $F$ as a subspace of $BC(\R,\X)$. Assume that {$(U(t,s))_{t\ge s}$ is an almost periodic evolutionary process generated by (\ref{eq0})} and, for each $h\ge 0$ and $g\in F$, the function $\R \ni t\mapsto U(t,t-h)g(t-h)$
belongs to $F$. Then, the \emph{evolution semigroup} $(T^h)_{h\ge 0}$ associated with (\ref{eq0}) in the function space $F$ is defined as the family of bounded operators $T^h, h\ge 0$, defined by
\begin{equation*}
[T^hg](t) =U(t,t-h)g(t-h), \quad g\in F, t\in \R , h\ge 0.
\end{equation*}
\end{definition}
From our assumption that $(U(t,s))_{t\ge s}$ is almost periodic evolutionary process, the function $\R \ni t\mapsto U(t,t-h)g(t-h)$ belongs to $\AP(\X)$ for every $g \in \AP(\X)$. Hence by choosing $F:= \AP(\X)$ in the worst case, we see that the evolution semigroup $(T^h)_{h\ge 0}$ associated with (\ref{eq0}) is well defined.

%%%%%%%%%%%%%%%%%%%%%%%%%%%%%%%%%%%%%%%%%%%%%%%%%%%%%%%%%%%%%%%%

\section{Asymptotic Behavior of Solutions}
Consider evolution equations of the form
\begin{equation}\label{eq1}
\dot x=[A_0+A(t)]x, \quad t\in \R , x\in \X ,
\end{equation}
where $A_0$ generates a $C_0$-semigroup denoted by $e^{tA_0},t\ge 0$, and $A:\R \to L(\X)$ is almost periodic in the norm topology.

To equation (\ref{eq1}) we associate the following integral equation
\begin{equation}\label{mild solution}
x(t)=e^{(t-s)A_0}x(s)+\int^t_s e^{(t-\xi )A_0} A(\xi )d\xi , \quad t\ge s; t,s\in \R .
\end{equation}
Every continuous function $x(\cdot )$ on an interval $J$, which is of the form $[a,b]$, $(a,b]$, $(a,b)$ or $[a,b)$, is said to be a mild solution of (\ref{eq}) on $J$.

It is well known that (\ref{eq1}) generates an evolutionary process in $\X$ which is determined by the integral equation (\ref{eq1}). The following theorem shows that this process is almost periodic, and its associated evolution semigroup $(T^h)_{h\ge 0}$ in $\AP(\X)$ leaves the function space $\AP_{\Lambda}(\X)$ invariant.

%%%%%%%%%%%%%%%%%%%%%%%%%%%%%%%%%%%%%%%%%%%%

\begin{theorem}\label{the invariance}
Under the above assumptions and notation, the evolution semigroup associated with equation \eqref{eq} in the function space $\AP_\Lambda (\X)$ is well defined as a $C_0$-semigroup.
\end{theorem}

%%%%%%%%%%%%%%%%%%%%%%%%%%%%%%%%%%%%%%%%%%%%

\begin{proof}
First, note that the evolution semigroup associated with the evolution equation $\dot x=Ax$ is well defined as a $C_0$-semigroup in $\AP_\Lambda (\X)$. In fact, this follows from the fact that $e^{hA}g(\cdot )$ is in $\AP(\X)$ whenever $g$ is in $\AP(\X)$. Moreover, $\sigma_b(e^{hA}g(\cdot ))\subset \sigma_b(g)$. Next, since $A(\cdot )$ is almost periodic in the norm topology, in $\AP(\X)$ we can define the operator $M_A$ of multiplication by $A(t)$, that is,
\begin{equation*}
M_A: \AP(\X) \ni g \mapsto A(\cdot )g(\cdot ) \in \AP(\X).
\end{equation*}
Note that  $M_A$ leaves $\AP_\Lambda (\X)$ invariant. In fact, this can be checked by using the Approximation Theorem
 of almost periodic functions \cite[Theorem 1.19, p.\ 27]{hinnaiminshi} as follows. Since $g\in \AP_\Lambda (\X)$, there is a sequence of trigonometric polynomials $g_n$ with exponents in $\Lambda$ that approximates $g$. Similarly, we construct a sequence of trigonometric  polynomials $A_n(\cdot)$ that approximates $A(\cdot )$ in norm topology with exponents also in $\Lambda$. Then, $A_n(\cdot )g_n(\cdot )$ approximates $A(\cdot )g(\cdot )$. As $\Lambda$ is the semi-module generated by $\sigma_b(A(\cdot ))$ we get that the exponents of
$A_n(\cdot )g_n(\cdot )$ lie in $\Lambda$.

Let $G_{A_0}$ be the generator of the evolution semigroup $(T^h_0)$ associated with $\dot x=A_0 x$ in $\AP_\Lambda (\X)$. Then, since $M_A$ is a bounded linear operator in $\AP_\Lambda (\X)$, $G_{A_0}+M_A$  generates a $C_0$-semigroup in $\AP_\Lambda (\X)$. We now show that this semigroup in $\AP_\Lambda (\X)$ is nothing but the evolution semigroup of (\ref{mild solution}) in $\AP_\Lambda (\X)$ associated with equation (\ref{eq1}). In fact, let us denote by $(S^h)$ the semigroup that is generated by $G_{A_0}+M_A$ in $\AP_\Lambda (\X)$. Then, this semigroup $(S^h)$ satisfies the equation
\begin{equation*}
S^h v=T^h_0v+\int^h_0 T^{h-\xi}_0M_A S^\xi vd\xi  , \quad \mbox{for all} \ h\ge 0, v\in \AP_\Lambda (\X).
\end{equation*}
Therefore, for each $t\in \R$, by the definition of the evolution semigroup $(T^h_0)$ associated with the equation $\dot x=A_0x$, we have
\begin{eqnarray*}
[S^hv](t) &=& (T_0^hv)(t) +\int^h_0 (T^{h-\xi}_0(M_A S^\xi v))(t)d\xi \\
&=& T_0(h)v(t-h) +\int^h_0 T_0(h-\xi )(M_A S^\xi v)(t -h+\xi )d\xi .
\end{eqnarray*}
Since $h$ and $t$ are arbitrary, we may set $h=t-s$, so the above equation becomes
\begin{eqnarray*}
[S^{t-s}v](t) &=&   T_0(t-s)v(s) +\int^h_0 T_0(t-s-\xi )(M_A S^\xi v)(s+\xi )d\xi     \\
&=& T_0(t-s)v(s) +\int^t_s T_0(t-\eta )(M_A S^{\eta -s} v)(\eta  )d\xi .
\end{eqnarray*}
Define $w(t):=[S^{t-s}v](t) $, then $w$ is the unique solution of the equation
\begin{eqnarray*}
w(t)&=&   T_0(t-s)v(s) +\int^h_0 T_0(t-s-\xi ) A(s+\xi ) S^\xi v(s+\xi )d\xi     \\
&=& T_0(t-s)x +\int^t_s T_0(t-\eta )A(\eta ) w(\eta )d\eta .
\end{eqnarray*}
However, this is the equation that defines $U(t,s)x$. Therefore, we have for all $t\ge s$, $v\in \AP_\Lambda (\X)$
\begin{equation*}
[S^{t-s}v](t)=U(t,s)v(s) .
\end{equation*}
In particular, when $h:=t-s$ we have that $S^hv=U(t,t-h)v(t-h)$ for all $h\ge 0, t\in \R, v\in \AP_\Lambda (\X)$, i.e., $(S^h)_{h\ge 0}$ is the evolution semigroup $(T^h)_{h\ge 0}$ for each $h\ge 0$. This completes the proof of the theorem.
\end{proof}

%%%%%%%%%%%%%%%%%%%%%%%%%%%%%%%%%%%%%%%%%%%%%%%%%%%

\begin{definition}
The evolution semigroup $(T^h)_{h\ge 0}$ associated with (\ref{eq1}) in the function space $\AP_{\Lambda}(\X)$ is called the {\it minimal evolution semigroup} associated with (\ref{eq1}) and $G := G_{A_0} + M_A$ is called the \emph{infinitesimal generator} of $T^h$.
\end{definition}
%%%%%%%%%%%%%%%%%%%%%%%%%%%%%%%%%%%%%%%%%%%%%%%%%%%

%%%%%%%%%%%%%%%%%%%%%%%%%%%%%%%%%%%%%%%%%%%%%%%%%%%
\begin{definition}
A well-posed equation Eq.(\ref{eq0}) is said to be {\it strongly stable} if the evolution process $(U(t,s))_{t\ge s}$ associated with it satisfies:
\begin{equation}
\lim_{t\to\infty} U(t,s)x =0
\end{equation}
for all fixed $x\in \X$ and $s\in \R$.
\end{definition}
If $A(t)$ is independent of $t$ and generates a $C_0$-semigroup, the strong stability of the evolution equation (\ref{eq0}) means that $\lim_{t\to \infty} T(t)x=0$ for all $x\in \X$.

\begin{theorem}\label{the 1}
Equation (\ref{eq1}) is strongly stable if its minimal evolution semigroup $(T^h)_{h\ge 0}$ is strongly stable.
\end{theorem}

%%%%%%%%%%%%%%%%%%%%%%%%%%%%%%%%%%%%%%%%%%%%%%%%%%%

\begin{proof}
Let $0\not= x_0\in\X$. Define $g(t)=x_0$ for all $t\in\X$. Obviously, $g\in \AP_\Lambda (\X)$ because $\sigma_b(g)\subset \{0\}$.
Since $(T^h)_{h\ge 0}$ is strongly stable,
\begin{equation*}
\lim_{h\to\infty} T^hz=0
\end{equation*}
for each $z\in \AP_\Lambda (\X)$. In particular,
\begin{equation*}
\lim_{h\to \infty}T^hg= 0.
\end{equation*}
That means,
\begin{equation*}
0 = \lim_{h\to \infty}\sup_{t\in\R} \| U(t,t-h)x_0\| \ge \lim_{h\to \infty} \| U(h,0)x_0\| \ge 0.
\end{equation*}
Therefore, every mild solution of (\ref{eq1}) is convergent to the origin, proving strong stability of (\ref{eq1}).
\end{proof}

%%%%%%%%%%%%%%%%%%%%%%%%%%%%%%%%%%%%%%%%%%%%%%%%%%%
\begin{definition}
Let us  denote by $\Sigma [(\ref{eq1})]$ the following set
\begin{equation*}
\sigma (G) \cap i\R =\{ \lambda = i\beta | \beta \in \R, i \beta \in \sigma (G)\} ,
\end{equation*}
and call it {\it the spectrum of equation} (\ref{eq1}).
\end{definition}

%%%%%%%%%%%%%%%%%%%%%%%%%%%%%%%%%%%%%%%%%%%%%%%%%%%

It can be checked (see e.g.\ \cite{naimin}) that the generator $G:= G_{A_0}+M_A$ is well-defined.
The domain $D(G)$ of $G$ consists of all functions $u$ in $\AP_\Lambda (\X)$ such that there exists a function $f\in \AP_\Lambda (\X)$ for which
\begin{equation*}
u(t)= U(t,s)u(s)+\int^t_s U(t,\xi )f(\xi )d\xi , \ \mbox{for all} \ t\ge s, t,s\in \R , f\in \AP_\Lambda (\X).
\end{equation*}
And, in this case, for such $f$ and $u$, $Gu=-f$. In the same way, $u\in D(G_{A_0})$ and $G_{A_0}u=-f$ if and only if for all $t \ge s$, $t,s\in \R$, $f\in \AP_\Lambda (\X)$
\begin{equation*}
u(t)= T_0(t-s)u(s)+\int^t_s T_0(t-\xi )f(\xi )d\xi .
\end{equation*}

The operator $G-\lambda$ generates a semigroup $(R^h)$ in $\AP_\Lambda(\X)$ which is uniquely determined by the equation
\begin{equation*}
R^hv=T^hv+\int^h_0 T^{h-\xi } (-\lambda ) R^\xi vd\xi , \quad h\ge 0, v\in \AP_\Lambda (\X).
\end{equation*}
Without difficulty we can check that $R^h=e^{-\lambda h}T^h$ which is exactly the evolution semigroup associated with the process $V(t,s):=e^{-\lambda (t-s)}U(t,s)$. Therefore,
$u\in D(G-\lambda )$ and $(G-\lambda )u=-f$ if and only if
\begin{equation*}
u(t)= e^{-\lambda (t-s)}U(t,s)u(s)+\int^t_s   e^{-\lambda (t-\xi) }  U(t,\xi )u(\xi )d\xi , \quad (t\ge s).
\end{equation*}

Therefore, $\lambda$ belongs to $\rho (G)$ if and only if for each $f\in \AP_\Lambda (\X)$ there is a unique solution $u_{\lambda ,f}$ to the equation
\begin{equation}\label{eq 3.16}
u_{\lambda ,f}(t)=e^{-\lambda (t-s)}U(t,s)u_{\lambda ,f}(s)+\int^t_s e^{-\lambda (t-\xi )}U(t,\xi )f(\xi )d\xi
\end{equation}
for all $t,s\in \R$ with $t \ge s$.

\bigskip
With this preparation we are ready to prove the main result of the paper.

%%%%%%%%%%%%%%%%%%%%%%%%%%%%%%%%%%%%%%%%%%%%%%%%%%%

\begin{theorem}\label{the main}
Assume that
\begin{equation}\label{con 2}
\sup_{t\ge s} \| U(t,s)\| <\infty,
\end{equation}
and for all $\lambda \in i\R$ but at most a countable set $\Sigma$, equation (\ref{eq 3.16}) has a unique solution $u_{\lambda ,f}$ for each given $f\in \AP_\Lambda (\X)$. Moreover, assume that for each $ \lambda \in \Sigma$ and each fixed $f\in \AP_\Lambda (\X)$,
\begin{equation}\label{ergo con}
\lim_{\alpha\downarrow 0}\alpha u_{\lambda +\alpha, f}=0.
\end{equation}
Then equation (\ref{eq1}) is strongly stable.
\end{theorem}
\begin{proof}
First, by (\ref{con 2}) the semigroup $(T^h)$ is uniformly bounded. By the Spectral Inclusion of $C_0$-semigroups (see Pazy \cite{pazy}), $\sigma (G) \subset \{ z\in \C : Re z \le 0\}$. Since $Re (\alpha + \lambda) = \alpha >0$, by Lemma \ref{added}, $\alpha + \lambda \in \rho (G)$, and thus, (\ref{ergo con}) makes sense.
Since $\Sigma$ is countable, also $\sigma _u(G, u)$ is countable.
The theorem is obtained by applying directly the ABLV Theorem \ref{the ABLV} to the evolution semigroup $(T^h)$ in $\AP_\Lambda (\X)$. In fact, as shown above, (\ref{ergo con}) is exactly the condition that
\begin{equation*}
\lim_{\alpha\downarrow 0}\alpha R(\alpha +\lambda ,G)u=0
\end{equation*}
for each $u\in \AP_\Lambda (\X)$, $\lambda \in \Sigma$. This means that the evolution semigroup $(T^h)$ is strongly stable in $\AP_\Lambda (\X)$. By Theorem \ref{the 1}, this yields the strong stability of (\ref{eq1}).
\end{proof}

\subsection*{Special cases of Theorem \ref{the main}}
Below we will discuss several special cases of the above theorem.
\begin{example}
If $A(t)=0$ for all $t$, then $\Lambda =\{ 0\}$. Therefore, $\AP_\Lambda (\X)$ is nothing but the space of all constant functions, hence it can be identified with $\X$. The evolution semigroup associated with (\ref{eq1}) is actually the semigroup $e^{tA_0}$ generated by the operator $A_0$.
\end{example}
Therefore, the following corollary is obvious, and is the ABLV Theorem.
\begin{corollary}
Let $A_0$ generate a uniformly bounded semigroup $T(t)$, and let $\sigma (A_0)\cap i\R $ be countable. Moreover, let for each $i\lambda\in \sigma (A_0)\cap i\R$ and $x\in \X$
\begin{equation*}
\lim_{\alpha \downarrow 0}\alpha R(i\lambda +\alpha ,A_0)x=0.
\end{equation*}
Then, equation (\ref{eq1}) is strongly stable.
\end{corollary}

\begin{example}
Consider the linear evolution equation
\[
\frac{dx}{dt} = A(t)x,
\]
where $x \in \X$, $A(t)$ is a (not necessarily bounded) linear operator acting on $\X$ for every fixed $t$ and $A(t+1) = A(t)$.
Then $\AP_\Lambda (\X)$ is nothing but the space of all $1$-periodic functions, and
the process $(U(t,s))_{t\ge s}$ is $1$-periodic (see \cite{naimin} for related concepts and results). By \cite[Proposition 1]{naimin}
equation (\ref{eq 3.16}) has a unique solution in $AP_\Lambda (\X)$ if and only if $1\in \rho ({e^{-\lambda} U(1,0)})$,  or equivalently, $e^{\lambda } \in \rho (U(1,0))$. Let us denote by $P$ the monodromy operator $U(1,0)$. It is well known that the strong stability of the $1$-periodic evolutionary process $(U(t,s))_{t\ge s}$ can be studied via the stability of its monodromy operator $P$ (see, for example, \cite{vu}). A discrete version of the ABLV Theorem on stability of individual orbits of the monodromy operator $P$ is given in \cite[Corollary 3.3]{min}. We now show that the conditions of \cite[Corollary 3.3]{min} actually yield the conditions of our Theorem \ref{the main}. In fact, the countability of $\Sigma[(\ref{eq1})]$ follows from the countability of $\sigma (P) \cap \Gamma$, where $\Gamma$ is the unit circle of the complex plane. Next, let $i\lambda_0 \in \Sigma (\ref{eq1})$. Then, $z_0:=e^{i\lambda_0} \in \sigma (P)\cap \Gamma$. Therefore,
\begin{eqnarray}\label{3.25}
\lim_{z\downarrow z_0}  (z-z_0)R(z,P)x_0 =0.
\end{eqnarray}
We will show that (\ref{3.25}) yields (\ref{ergo con}). By the definition of limit (\ref{3.25}) in \cite{min}, (\ref{3.25}) can be re-written as
\begin{eqnarray*}
\lim_{\alpha \downarrow 0}  (e^{i\lambda_0+\alpha}-e^{i\lambda_0})R(e^{i\lambda_0 +\alpha},P)x_0 =0.
\end{eqnarray*}
Or equivalently,
\begin{eqnarray*}
\lim_{\alpha \downarrow 0}  (e^{\alpha}-1)R(e^{i\lambda_0 +\alpha},P)x_0 =0,
\end{eqnarray*}
and hence,
\begin{eqnarray*}
\lim_{\alpha \downarrow 0} \alpha R(e^{i\lambda_0 +\alpha},P)x_0 =0.
\end{eqnarray*}
Let $f\in AP_{\Lambda }(\X)$, that is, $f$ is an arbitrary continuous $1$-periodic function taking values in $\X$. Then, let
\begin{equation*}
x_0:= \int^1_0 e^{-(i\lambda_0+\alpha)(1-\xi )} U(1,\xi )f(\xi ) d\xi .
\end{equation*}
And let $u_{i\lambda_0+\alpha ,f}$ be the unique solution to equation (\ref{eq 3.16}). Then, it is easy to check that
\begin{equation*}
u_{i\lambda_0+\alpha ,f}(0) =  R(e^{i\lambda_0 +\alpha},P)x_0.
\end{equation*}
Therefore, by the definition of the evolutionary process,
\begin{eqnarray*}
\| u_{i\lambda_0+\alpha ,f}\|  :=  \sup_{t\in \R }\| u_{i\lambda_0+\alpha ,f}(t)\| & =& \sup_{0\le t\le 1}  \| u_{i\lambda_0+\alpha ,f}(t)\| \nonumber \\
&=& Me^\beta   \| u_{i\lambda_0+\alpha ,f}(0)\| \nonumber \\
&\le &  Me^\beta \| R(e^{i\lambda_0 +\alpha},P)x_0\| ,
\end{eqnarray*}
where the positive numbers $M,\beta$ depend only on the process $(U(t,s))_{t\ge s}$ (see Definition \ref{def evpr}). This shows that (\ref{3.25}) yields (\ref{ergo con}), and thus, the following result is a corollary to Theorem \ref{the main}.
\begin{corollary}\label{lem 3.3}
Let the monodromy operator $P$ of the $1$-periodic evolutionary process $(U(t,s))_{t\ge s}$ be a power bounded
operator, i.e.\ $\sup_{n \in \N}\|P^n\| < \infty$,
such that $\sigma (P)\cap \Gamma$ is a countable set.
Moreover, assume that for each $\xi_0\in \sigma (P)\cap \Gamma$ the following holds for each $x_0\in \X$
\begin{equation*}
\lim_{\lambda  \downarrow \xi_0} (\lambda -\xi_0 )R(\lambda ,  P)x_0 =0 .
\end{equation*}
Then, for every $x_0\in\X$ and for every $s\in \R$,
\begin{equation*}
\lim_{t\to\infty} U(t,s)x_0=0 .
\end{equation*}
\end{corollary}
\end{example}

\medskip
In summary, Theorem \ref{the main} covers two well known special cases of non-autonomous equations, including the autonomous and periodic cases. For the general case of non-autonomous equations the generator ${G}$ of the evolution semigroup may have a more complicated spectrum, and we will discuss this topic in the next section.

\section{Analysis of the Spectrum of the Generator ${ G}$}
As shown in the previous section, the spectrum of the generator $G$ of the evolution semigroup $(T^h)_{h\ge 0}$ plays an important role in studying the asymptotic behavior of the equations. Moreover, in the autonomous and periodic cases this spectrum may not be the whole vertical strips.
In this section we will give a detailed analysis of the spectrum of the generator ${G}$ of the evolution semigroup associated with certain non-periodic equations and applications of  the results obtained from the previous section.

\begin{proposition}\label{semimodule=module_case}
Let $(T^h)$ be the minimal evolution semigroup associated with (\ref{eq1}), and $G$ be its generator. Assume further that
the semi-module generated by the set of frequencies of the function $A(\cdot )$ is actually a module.
Then, for each $\lambda \in \Lambda$
\begin{eqnarray*}
i\lambda +\sigma (G) &\subset& \sigma (G)\\
i\lambda +\rho (G) &\subset& \rho (G).
\end{eqnarray*}
\end{proposition}
\begin{proof}
Let $\lambda \in \Lambda$. The above inclusions are actually equivalent to the claim that $\mu\in \rho (G)$ if and only if $i\lambda +\mu \in \rho (G)$. By the argument that precedes Theorem \ref{the main}, $\mu\in \rho (G)$ if and only if for each $f\in \AP_\Lambda (\X)$ the integral equation
\begin{equation}\label{4.3}
x(t)=e^{-\mu (t-s)}U(t,s)x(s)+\int^t_s e^{-\mu (t-\xi )}U(t,\xi )f(\xi )d\xi \quad (t\ge s)
\end{equation}
has a unique solution in $\AP_\Lambda (\X)$, denoted by $u_{\mu ,f}$. Since $\Lambda$ is assumed to be a module, $u_{\mu ,f}\in \AP_\Lambda (\X)$ if and only if the function $v:\R \to \R, t\mapsto e^{i\lambda t}u_{\mu ,f}$, belongs to $\AP_\Lambda (\X)$. Moreover,
$u_{\mu ,f}$ is the unique solution of (\ref{4.3}) if and only if $v$ is the unique solution of the equation
\begin{equation*}
y(t)=e^{-(i\lambda +\mu) (t-s)}U(t,s)y(s)+\int^t_s e^{-(i\lambda +\mu )(t-\xi )}U(t,\xi )f(\xi )d\xi \quad (t\ge s),
\end{equation*}
that is, $i\lambda +\mu \in \rho (G)$.
\end{proof}

%%%%%%%%%%%%%%%%%%%%%%%%%%%%%%%%%%%%%%%%%%%%%%%%%%%

\begin{remark}
As an example for the case when $\Lambda$ is a module we can take the equation $\dot x (t) = (e^{it}+e^{-it})x(t)$. Then, the semi-module generated by the set of frequencies is the set of all reals of the form
$$
n\cdot 1 +m \cdot (-1),
$$
where $m,n$ are non-negative integers, hence $\Lambda = \Z$. In the one dimensional case, $sm (\sigma_b(a)) = m (\sigma_b(a))$ if $a(t)$ is a periodic function and its spectrum contains both positive and negative parts, or if $a(t)$ is an almost periodic function that has a symmetric spectrum.
\begin{proof}
If $a(t)$ is $L$-periodic function with Fourier series
\[
 \sum_{k \in \Z} C_k e^{2\pi ikx/L}
\]
then $ \sigma_b (a) = \{2\pi k/L | C_k \ne 0\}$ (see \cite[Example 1.3, p.24]{hinnaiminshi}).
Suppose $k_i < 0 < k_j$ such that $C_{k_i} \ne 0, C_{k_j} \ne 0$. Let $\lambda = \frac{2 \pi}{L} (n_1 k_1 + n_2 + \cdots n_p k_p)$ be arbitrary element of $m (\sigma_b(a))$. Since $sm (k_i, k_j) = m (k_i, k_j)$, $$sm (k_1, k_2, \cdots, k_p, k_i, k_j) = m (k_1, k_2, \cdots, k_p, k_i, k_j).$$
So $\lambda \in \frac{2\pi}{L}. m (k_1, k_2, \cdots, k_p, k_i, k_j) = sm (k_1, k_2, \cdots, k_p, k_i, k_j) \subset sm (\sigma_b (a))$. It means $m (\sigma_b(a)) \subset sm (\sigma_b (a))$ and hence $m (\sigma_b(a)) = sm (\sigma_b (a))$. The rest case is quite trivial.
\end{proof}If $\Lambda$ is not a module, the situation may be more complicated.
\end{remark}

%%%%%%%%%%%%%%%%%%%%%%%%%%%%%%%%%%%%%%%%%%%%%%%%%%%

In order to analyze the spectrum of the generator $G$ of the minimal evolution semigroup associated with (\ref{eq1}), in case the non-autonomous term $A(t)$ is small, it is useful to consider the generator $G_{0,\Lambda} $ of the evolution semigroup $(T_{0,\Lambda} ^h)$ associated with the equation $\dot u=A_0u$ in the function space $\AP_\Lambda (\X)$.
\begin{lemma}
Assume that $A_0$ is a sectorial operator, and the semi-module $\Lambda$ is a closed subset of the real line. Under the above assumption and notation,
\begin{equation*}
\sigma (G_{0,\Lambda}) = \sigma (A_0) - i\Lambda .
\end{equation*}
\end{lemma}
\begin{proof}
By the main results of \cite{murnaimin},
\begin{equation*}
\mu \in \rho (G_{0,\Lambda }) \Leftrightarrow \sigma (A_0-\mu) \cap i\Lambda =\emptyset .
\end{equation*}
This condition means there are no complex numbers $\eta \in \sigma (A_0)$ and $\lambda \in \Lambda$ such that $\eta -\mu =i\lambda$, or, $\mu$ cannot be expressed as $\mu =\eta -i\lambda$ with $\eta \in \sigma (A_0)$ and $\lambda \in \Lambda$. In turn, this yields that $\mu\not\in  \sigma (A_0) -i\Lambda $, or, $\mu \in \C \backslash (\sigma (A_0) -i\Lambda )$.
This proves the proposition.
\end{proof}
Recall that we are denoting by $G$ the generator of the evolution semigroup associated with equation (\ref{eq1}).
\begin{proposition}
Assume that $A_0$ is a sectorial operator, and the semi-module $\Lambda$ is a closed subset of the real line. Then, for each compact subset
$$K\subset \rho (G_{0,\Lambda}) =\C \backslash (\sigma (A_0) - i\Lambda)$$ there exists a number $\delta_0>0$ such that if
\begin{equation*}
\sup_{t\in\R} \| A(t)\| <\delta_0,
\end{equation*}
then,
\begin{equation}\label{4.8}
\sigma (G) \cap  K= \emptyset .
\end{equation}
\end{proposition}

%\begin{remark}
%The condition $\sigma (G) \cap  K= \emptyset \Leftrightarrow \sigma (G) \subset \overline{K} \approx \sigma (A_0) - %i\Lambda$. It means that the spectrum of the generator $G$ is ''close'' to $\sigma (A_0) - i\Lambda$ which is a ''discrete %countable set''. In the case of one dimensional below, we will prove
%\[
%\sigma (G) \subset -i \Lambda \cup i \Lambda.
%\]
%\end{remark}

\begin{proof}
Since $K$ is a compact subset of $\rho (G_{0,\Lambda}) $
\begin{equation*}
\sup_{\lambda \in K} \| R(\lambda , G_{0,\Lambda})\| =\mu <\infty .
\end{equation*}
Let ${\cal A}$ denote the operator of multiplication by $A(t)$ in $\AP_\Lambda (\X)$. Note that this multiplication operator is well defined in $\AP_\Lambda (\X)$, and moreover, this operator is bounded. Therefore, there exists a positive $\delta_0$ such that the operator
\begin{equation*}
(I-R(\lambda ,G_{0,\Lambda}){\cal A})^{-1}
\end{equation*}
whenever $\| {\cal A}\| <\delta_0$ and $\lambda \in K$. Next, we can show that for each $\lambda \in K$
\begin{equation*}
(I-R(\lambda ,G_{0,\Lambda}){\cal A})^{-1} R(\lambda , G_{0,\Lambda}) = R(\lambda , G_{0,\Lambda }+{\cal A})=R(\lambda , G).
\end{equation*}
In fact, set
$$
U:= (I-R(\lambda ,G_{0,\Lambda}){\cal A})^{-1} R(\lambda , G_{0,\Lambda}) .
$$
With this notation we have
\begin{equation*}
(I-R(\lambda ,G_{0,\Lambda }){\cal A}) U = R(\lambda , G_{0,\Lambda }).
\end{equation*}
Therefore,
$$
U=R(\lambda,G_{0,\Lambda}) + R(\lambda,G_{0,\Lambda}) {\cal A}U,
$$
and hence,
$$
(\lambda -G_{0,\Lambda})U=I+{\cal A}U.
$$
This yields
$$
(\lambda -G_{0,\Lambda})U-{\cal A}U=I,
$$
that is,
\begin{equation*}
(\lambda -G_{0,\Lambda}-{\cal A})U=(\lambda -G)U=I.
\end{equation*}
In other words, $U=R(\lambda , G)$ whenever $\lambda \in K$ and $\|{\cal A}\| <\delta_0$. This yields  (\ref{4.8}). The proposition is proved.
\end{proof}

%%%%%%%%%%%%%%%%%%%%%%%%%%%%%%%%%%%%%%%%%%%%%%%%%%%

\subsection{One dimensional case}

As a motivation, we consider the equation
\begin{equation*}
\frac{dx}{dt}= a(t)x,
\end{equation*}
where $a(t)$ is a {\em numerical} almost periodic function taking values in $\X := \R$. Define the operator $G= -d/dt +a(t)$. We will describe the part of spectrum $\sigma (G)\cap i\R$.

%%%%%%%%%%%%%%%%%%%%%%%%%%%%%%%%%%%%%%%%%%%%%%%%%%%

\begin{theorem}\label{new}
Suppose that
\begin{itemize}
\item[(i)] the semi-module generated by the Bohr spectrum of $a(t)$, i.e. $\Lambda := sm (\sigma_b (a))$, is a discrete countable set and
\item[(ii)] $0 \not \in \sigma_b (a)$.
\end{itemize}
Then  $\Sigma$
is also a discrete countable set. Moreover
\[
\Sigma \subset -i \Lambda \cup i\Lambda.
\]
\end{theorem}

%%%%%%%%%%%%%%%%%%%%%%%%%%%%%%%%%%%%%%%%%%%%%%%%%%%

In order to prove Theorem \ref{new}, we need the following two lemmas.

%%%%%%%%%%%%%%%%%%%%%%%%%%%%%%%%%%%%%%%%%%%%%%%%%%%

%%%%%%%%%%%%%%%%%%%%%%%%%%%%%%%%%%%%%%%%%%%%%%%%%%%

\begin{lemma}\label{Theorem4.12_Fink}
Let $f \in \AP (\R)$ such that $| \lambda_n | \ge M > 0$ for all $\lambda_n \in \sigma_b (f)$. Then $g(\cdot) := \int_0^{\cdot} f(s)ds \in \AP (\R)$ and $\sigma_b (g) \subset \sigma_b (f) \cup \{0\}$.
\end{lemma}

%%%%%%%%%%%%%%%%%%%%%%%%%%%%%%%%%%%%%%%%%%%%%%%%%%%

\begin{proof}

This Lemma is a direct consequence of \cite[Theorem 4.12]{Fink} and \cite[Theorem 5.2]{Fink}. Indeed, \cite[Theorem 4.12]{Fink} stated that $g(t) = \int_0^t f(s)ds$ is an almost periodic function. Therefore $g(t)$ is an almost periodic solution to equation $x'(t) = f(t)$, and by \cite[Theorem 5.2]{Fink}, $\sigma_b (g) \subset \sigma_b (f) \cup \{0\}.$
\end{proof}

%%%%%%%%%%%%%%%%%%%%%%%%%%%%%%%%%%%%%%%%%%%%%%%%%%%

\begin{lemma}\label{new2}
$e^{\int_0^{\cdot} a(s) ds} \in \AP _{\Lambda}(\R)$.
\end{lemma}

%%%%%%%%%%%%%%%%%%%%%%%%%%%%%%%%%%%%%%%%%%%%%%%%%%%

\begin{proof} First, note that if we replace $sm(\sigma_b (a))$ with $m(\sigma_b (a))$ - the module generated by $\sigma_b (a)$, then the claim is clear from \cite[Theorem 1.9, p.5]{Fink}. However, in general $sm(\sigma_b (a))$ may differ from $m(\sigma_b (a))$.
Since $sm (\sigma_b (a))$ is a discrete set, so is $\sigma_ b(a)$. From the assumption $0 \not \in \sigma_b (a)$, we get that $\sigma_b (a)$ is bounded away from zero. Applying Lemma \ref{Theorem4.12_Fink},
\[
g(\cdot) := \int_0^{\cdot} a(t)dt \in \AP_{\Lambda} (\R).
\]
 We have
\[
e^{g(t)} = 1 + g(t) + \frac{g^2 (t)}{2} + \frac{g^3 (t)}{3!} + \cdots + \frac{g^n(t)}{n!} + \cdots
\]
Since $g \in \AP _{\Lambda} (\R)$ we have that $g$ is bounded, so the above infinite sum is uniformly convergent.
Also $g^n \in \AP _{\Lambda} (\R)$ due to the fact that $\AP_{\Lambda} (\R)$ is closed under products (see \cite[Theorem 1.9, p.\ 5]{Fink}) and $\sigma_b(g^n) \in \Lambda$ as an application of the Approximation Theorem.
Since the uniform limit of a sequence of almost periodic functions is also an almost periodic function, $e^{g(\cdot)} \in \AP (\R)$. On the other hand, $\AP _{\Lambda}(\R )$ is a closed subspace of $\AP (\R)$. Therefore, $e^{g(\cdot)} \in \AP_{\Lambda} (\R)$.
\end{proof}

%%%%%%%%%%%%%%%%%%%%%%%%%%%%%%%%%%%%%%%%%%%%%%%%%%%

\begin{proof}[Proof of theorem \ref{new}]
It is sufficient to prove that for every real number $\lambda \not \in -\Lambda \cup \Lambda$, the equation
\begin{equation}\label{one.dim.eq1}
\frac{dx}{dt} = (a(t) -i\lambda) x +f(t)
\end{equation}
has a unique solution $x \in \AP _{\Lambda} (\R)$ for every $f \in \AP_{\Lambda} (\R)$.

Let $y(t) = e^{i \lambda t} x(t)$ or $x(t) = e^{-i\lambda t} y(t)$, then (\ref{one.dim.eq1}) becomes
\begin{equation}\label{eq_new}
\frac{dy}{dt} =  a(t) y + e^{i \lambda t} f(t),
\end{equation}
which has a general solution
\[
y(t) = e^{\int_0^t a(s) ds}  \left ( \int_0^t e^{i \lambda \tau} f(\tau) e^{-\int_0^{\tau}a(s) ds} d\tau + C_0 \right ).
\]
By applying Lemma \ref{new2}, one has $f(\cdot) e^{-\int_0^{\cdot}a(s) ds} \in \AP _{\Lambda}(\R)$. Now suppose that
\[
\lambda_n \in \sigma_b \left (e^{i \lambda \tau} f(\tau) e^{-\int_0^{\tau}a(s) ds}\right ),
\]
then $\lambda_n = \lambda + \mu_n$ for some $\mu_n \in \Lambda$. Since $\lambda \not \in -\Lambda \cup \Lambda$, it follows that $\lambda _n \ne 0$ for all $n$. Because $\Lambda$ is a discrete set, $e^{i \lambda \tau} f(\tau) e^{-\int_0^{\tau}a(s) ds}$ is an almost periodic function with Bohr spectrum bounded away from zero.

By applying Lemma \ref{Theorem4.12_Fink}, it follows that $G(\cdot) = \int_0^{\cdot} e^{i \lambda \tau} f(\tau) e^{-\int_0^{\tau}a(s) ds} d\tau \in \AP (\R)$. In fact $G \in \AP_{\Lambda + \{\lambda\}} (\R)$, where $\Lambda + \{\lambda\}$ is the set of all real numbers $\mu$ of the form
\[
\mu := m + \lambda, \quad m \in \Lambda.
\]
Note that $\Lambda + \{\lambda\} \not = \Lambda$, since $\lambda \not \in \Lambda$.
 So $y \in \AP_{\Lambda + \{\lambda\}} (\R)$ and therefore $x(\cdot) = e^{-i\lambda \cdot} y(\cdot) \in \AP _{\Lambda} (\X)$  which completes the proof.
\end{proof}

%%%%%%%%%%%%%%%%%%%%%%%%%%%%%%%%%%%%%%%%%%%%%%%%%%%

\begin{remark}
The condition $0 \not \in \sigma_b (a)$ is essential. Otherwise the solution
\[
y(t) = e^{\int_0^t a(s) ds}  \left ( \int_0^t e^{i \lambda \tau} f(\tau) e^{-\int_0^{\tau}a(s) ds} d\tau + C_0 \right )
\]
is even unbounded. For instance, choose $a(t) := 1 + e^{it}$, then $\int_0^t a(s) ds = t + \frac{e^{it}}{i}$ is unbounded.
\end{remark}

%%%%%%%%%%%%%%%%%%%%%%%%%%%%%%%%%%%%%%%%%%%%%%%%%%%

\begin{example}
Consider the equation
\begin{equation}\label{example}
\frac{dx(t)}{dt}=(e^{it}+e^{i\sqrt{2}t})x(t), \quad x(t)\in \C, t\in \R .
\end{equation}
The frequencies of $a(t):=(e^{it}+e^{i\sqrt{2}t})$ are $\{ 1,\sqrt{2}\}$. The semi-module generated by the set of frequencies of $a(t)$ is the set $\N_0 +\sqrt{2}\N_0$. With the operator ${G}= -d/dt +a(t)$ one gets
\[
 \sigma (G) \cap i\R \subset -i (\N _0 + \N_0 \sqrt{2}) \cup i (\N _0 + \N_0 \sqrt{2}),
\]
where $\N_0 = \{0, 1, 2, \ldots\}$. In this example $U(t, s) = e^{\int_s^t a(\tau )d\tau }$, so all solutions $x(t) = U(t, s) x(s)$ are almost periodic, and therefore bounded.

\end{example}

\begin{remark}
 The discreteness of $sm (\sigma_b (a))$ is essential. In one dimensional case, the semi-module generated by the Bohr spectrum of all periodic functions are discrete. This assertion also holds for all almost periodic functions that either have discrete positive spectrums or discrete negative spectrums; for example
\[
 a(t) = c_1 e^{i \lambda_1} + c_2 e^{i\lambda_2} + \cdots + c_m e^{i\lambda_m},
\]
where $\lambda_1 < \lambda_2 < \cdots < \lambda_m < 0$ or $0 < \lambda_1 < \lambda_2 < \cdots < \lambda_m$.
Meanwhile with almost periodic functions that have both negative and positive spectrums, their semi-module generated by the Bohr spectrums may be not discrete. For example, if
\[
 a(t) = e^{-it} + e^{i\sqrt{2}t},
\]
then $sm (\sigma_b(a)) = \{-m + n\sqrt{2}| m, n \in \N_0\}$ is not discrete since $-m + n\sqrt{2}$ can be closed to zero with arbitrary distance.

\end{remark}
\begin{remark}
 Theorem \ref{new} also holds if $\X$ is any Banach space. In fact, Lemma \ref{Theorem4.12_Fink} and Lemma \ref{new2} work well in this case.
\end{remark}

\begin{remark}
 Theorem \ref{new} holds not only for one dimensional case, but also for finite dimensional case, i.e, for equation
\[
 \frac{dx}{dt} = A(t)x,
\]
where $A(t)$ is an almost periodic matrix and $\X = \R^n$ or $\C^n$.
\end{remark}

\begin{example}
 In the following we give a numerical example to illustrate the conditions required in our Theorem \ref{the main}. The equation we consider is of the form
\begin{equation}\label{example2}
 \frac{dx(t)}{dt} = (\cos t + \cos t\sqrt{2} - 2)x(t), \quad x(t)\in \R , t\in \R_+.
\end{equation}
\end{example}

The frequencies of $a(t):=\cos t + \cos t\sqrt{2} - 2$ are $\{0, -1, 1,-\sqrt{2}, \sqrt{2}\}$, therefore $\Lambda := sm (\sigma_b (a)) = m (\sigma_b (a)) = \Z + \Z \sqrt{2}$. We will show that 
\[
 \Sigma = i\Lambda.
\]

We have $M_{0,a} := \lim\limits_{T \to \infty}\frac{1}{T} \int_0^T a(t)dt = -2$, so $\re M_{0,a} = -2 \not = 0$. It is shown in \cite[Theorem 6.6]{Fink} that for all $\lambda \not \in \Lambda$, the equation 
\[
 \frac{dx}{dt} = (a(t) -i\lambda )x + f(t) ,
 \]
 or with $ y = e^{i\lambda t}x$,
 \[
 \frac{dy}{dt} = a(t)y + e^{i\lambda t}f(t) 
\]

 has a unique bounded solution
\[
 y(t) = -\int_t^{\infty} e^{\int_s^t a(\tau ) d\tau} e^{i\lambda s}f(s)ds 
\]
which is almost periodic and in $\AP_{\Lambda + \lambda}(\X)$. It yields that $\Sigma \subset i\Lambda$. 
We now verify the condition (\ref{ergo con}). Since $U(t, s) = e^{\int_s^t a(\tau ) d\tau}$, (\ref{eq 3.16}) becomes (for $s=0$)
\[
 u_{\lambda, f} = e^{-\lambda t} e^{\int_0^t a (\tau ) d\tau} u(0) + \int_0^t e^{-\lambda (t-\xi)}e^{\int_{\xi}^t a(\tau )d\tau} f(\xi) d\xi.
\]
Therefore for each $i\lambda \in \Sigma$,
\begin{equation*}
\begin{split}
 u_{\alpha + i\lambda, f} & = e^{-(\alpha + i\lambda) t} e^{\int_0^t a (\tau ) d\tau} u(0) + \int_0^t e^{-(\alpha + i\lambda) (t-\xi)}e^{\int_{\xi}^t a(\tau )d\tau} f(\xi) d\xi \\
& = e^{-(2+\alpha)t}e^{-i\lambda t + \sin t + \frac{\sin \sqrt{2}t}{\sqrt{2}}} \left [ u(0) + \int_0^t e^{i\lambda \xi - \sin \xi - \frac{\sin \sqrt{2}\xi}{\sqrt{2}}}f(\xi)e^{2\xi} d\xi \right ]
\end{split}
\end{equation*}
Since $e^{-i\lambda t + \sin t + \frac{\sin \sqrt{2}t}{\sqrt{2}}}$ and $e^{i\lambda \xi - \sin \xi - \frac{\sin \sqrt{2}\xi}{\sqrt{2}}}f(\xi)$ are almost periodic in  $t$ and $\xi$, respectively, there exist positive constants $M$ and $N$ such that
\[
 |e^{-i\lambda t + \sin t + \frac{\sin \sqrt{2}t}{\sqrt{2}}}| \le M,\quad |e^{i\lambda \xi - \sin \xi - \frac{\sin \sqrt{2}\xi}{\sqrt{2}}}f(\xi)| \le N.
\]
We have
\begin{equation*}
 \begin{split}
 |u_{\alpha + i\lambda, f}| & \le e^{-(2+\alpha)t}M \left [u_0 + N \int_0^t e^{2\xi}d\xi \right ] \\
& \le e^{-(2+\alpha)t}M u_0 +  e^{-(2+\alpha)t}MN\frac{e^{2t}-1}{2} \\
&\le Mu_0 + \frac{MN}{2}(e^{-\alpha t} - e^{-(2+\alpha)t}) \\
& \le Mu_0 + 2MN.
 \end{split}
\end{equation*}
Therefore $\lim\limits_{\alpha \downarrow 0}\alpha u_{\alpha + i\lambda, f} = 0$. Obviously, the set $\Sigma$ is countable, so according to Theorem \ref{the main}, the equation (\ref{example2}) is strongly stable.

\section*{Acknowledgments}

The first author was supported by Vietnam MOET grant No. 322/5179/QD-BGDDT. The second author was supported by DFG under grant number Si801/6-1 and NAFOSTED under grant number 101.02-2011.47.

%%%%%%%%%%%%%%%%%%%%%%% REFERENCES %%%%%%%%%%%%%%%%%%%%%%%%%%
\bibliographystyle{amsplain}

\end{document}